\documentclass[journal]{IEEEtran}
\usepackage{amsmath}
\usepackage{mathtools}
\usepackage{amsthm}
\usepackage{amsfonts}
\usepackage{verbatim}
\usepackage{mathrsfs}
\usepackage{url}
\usepackage[english]{babel}

\usepackage{amssymb,latexsym,amsmath}

\usepackage{epsfig}
\usepackage{epstopdf}

\usepackage{caption}

\def\mindex#1{\index{#1}}

\def\sq{\hbox{\rlap{$\sqcap$}$\sqcup$}}
\def\qed{\ifmmode\sq\else{\unskip\nobreak\hfil
\penalty50\hskip1em\null\nobreak\hfil\sq
\parfillskip=0pt\finalhyphendemerits=0\endgraf}\fi\medskip}

\long\def\defbox#1{\framebox[.9\hsize][c]{\parbox{.85\hsize}{%
\parindent=0pt
\baselineskip=12pt plus .1pt      
\parskip=6pt plus 1.5pt minus 1pt 
 #1}}}

\long\def\beginbox#1\endbox{\subsection*{}%
\hbox{\hspace{.05\hsize}\defbox{\medskip#1\bigskip}}%
\subsection*{}}

\def\endbox{}

\newsavebox{\junk}
\savebox{\junk}[1.6mm]{\hbox{$|\!|\!|$}}

\def\limsup{\mathop{\rm lim\ sup}}
\def\liminf{\mathop{\rm lim\ inf}}

\def\state{{\mathbb X}}

\def\bx{{{\cal B}(\mathbb{X})}}

\newcommand{\field}[1]{\mathbb{#1}}

\def\Re{\field{R}}

\def\bfmath#1{{\mathchoice{\mbox{\boldmath$#1$}}%
{\mbox{\boldmath$#1$}}%
{\mbox{\boldmath$\scriptstyle#1$}}%
{\mbox{\boldmath$\scriptscriptstyle#1$}}}}

\def\bfmY{\bfmath{Y}}

\def\bfmhhaY{\bfmath{\hhaY}} 
\def\bfmhhaY{\hbox to 0pt{$\widehat{\bfmY}$\hss}\widehat{\phantom{\raise 1.25pt\hbox{$\bfmY$}}}}

\def\til={{\widetilde =}}

\def\clB{{\cal B}}

\def\clF{{\cal F}}

\def\clM{{\cal M}}

 \def\FRAC#1#2#3{\genfrac{}{}{}{#1}{#2}{#3}}

\def\ddtp{{\mathchoice{\FRAC{1}{d^{\hbox to 2pt{\rm\tiny +\hss}}}{dt}}%
{\FRAC{1}{d^{\hbox to 2pt{\rm\tiny +\hss}}}{dt}}%
{\FRAC{3}{d^{\hbox to 2pt{\rm\tiny +\hss}}}{dt}}%
{\FRAC{3}{d^{\hbox to 2pt{\rm\tiny +\hss}}}{dt}}}}

\def\half{{\mathchoice{\FRAC{1}{1}{2}}%
{\FRAC{1}{1}{2}}%
{\FRAC{3}{1}{2}}%
{\FRAC{3}{1}{2}}}}

\def\eqdef{\mathbin{:=}}

\def\Prob{P}

\def\Expect{E}

\def\average#1,#2,{{1\over #2} \sum_{#1}^{#2}}

\def\eye(#1){{\bf(#1)}\quad}

\newtheorem{theorem}{Theorem}[section]
\newtheorem{corollary}[theorem]{Corollary}
\newtheorem{proposition}[theorem]{Proposition}
\newtheorem{lemma}[theorem]{Lemma}

\def\Theorem#1{Theorem~\ref{#1}}

\def\eq#1/{(\ref{e:#1})}

\newcommand{\beqn}[1]{\notes{#1}%
\begin{eqnarray} \elabel{#1}}

\newcommand{\eeqn}{\end{eqnarray} }

\newcommand{\beq}[1]{\notes{#1}%
\begin{equation}\elabel{#1}}

\newcommand{\eeq}{\end{equation}}

\def\bdes{\begin{description}}
\def\edes{\end{description}}

\newcounter{rmnum}
\newenvironment{romannum}{\begin{list}{{\upshape (\roman{rmnum})}}{\usecounter{rmnum}
\setlength{\leftmargin}{14pt}
\setlength{\rightmargin}{8pt}
\setlength{\itemindent}{-1pt}
}}{\end{list}}

\newcounter{anum}

{\end{list}}

\def\ass(#1:#2){(#1\ref{#1:#2})}

\def\ritem#1{
\item[{\sf \ass(\current_model:#1)}]
}

\newenvironment{recall-ass}[1]{%
\begin{description}
\def\current_model{#1}}{
\end{description}
}

\newcommand{\bd}{\begin{description}}
\newcommand{\ed}{\end{description}}
\newcommand{\bt}{\begin{theorem}}
\newcommand{\et}{\end{theorem}}
\newcommand{\ba}{\begin{array}{rcl}}
\newcommand{\ea}{\end{array}}

\allowdisplaybreaks

\def\tau{{\cal T}}

\newlength{\noteWidth}
\long\def\notes#1{\ifinner
           {\tiny #1}
           \else
           \marginpar{\parbox[t]{\noteWidth}{\raggedright\tiny #1}}
       \fi\typeout{#1}}

\newtheorem{definition}{Definition}[section]
\newtheorem{fact}{Fact}[section]
\newtheorem{remark}{Remark}[section]
\newtheorem{assumption}{Assumption}[section]

\begin{document}

\title{On Rates of Convergence for Markov Chains under Random Time State Dependent Drift Criteria}
\author{Ramiro Zurkowski, Serdar Y\"uksel, Tam\'as Linder
\thanks{The authors are with the Department of Mathematics and Statistics, Queen's University,
  Kingston, Ontario, Canada, K7L 3N6. This research was partially supported by
  the
  Natural Sciences and Engineering Research Council of Canada (NSERC). Email: 8raz@queensu.ca,yuksel@mast.queensu.ca,linder@mast.queensu.ca.}}
\maketitle

\begin{abstract}
  Many applications in networked control require intermittent access of a
  controller to a system, as in event-triggered systems or information
  constrained control applications. Motivated by such applications and extending previous work on Lyapunov-theoretic drift criteria, we establish both
  subgeometric and geometric rates of convergence for Markov chains under
  state dependent random time drift criteria. We quantify how the rate of
  ergodicity, nature of Lyapunov functions, their drift properties, and the
  distributions of stopping times are related. We finally study an application
  in networked control.
\end{abstract}

\section{Introduction and literature review}
Stochastic stability of Markov chains has an almost complete theory, and forms a
foundation for several other general techniques such as dynamic programming,
linear programming approach to Markov Decision Processes \cite{Borkar2},  and
Markov Chain Monte-Carlo (MCMC) \cite{CTCN}. One powerful approach to establish
stochastic stability is through single-stage (Foster-Lyapunov) drift criteria
\cite{MeynBook}. The \textit{state-dependent} criteria
\cite{MeynStateDrift,malmen79,ConnorFort} relax the one-stage criteria to
criteria involving time instances which are state-dependent but
deterministic. Such criteria form the basis of the fluid-model (or ODE) approach
to stability in stochastic networks and other general models
\cite{bormey00a}, \cite{formeymoupri08a}, \cite{dai95a}, \cite{daimey95a}, \cite{CTCN}. Building
on \cite{MeynBook} and \cite{MeynStateDrift}, \cite{YukMeynTAC2010} considered
stability criteria based on a state-dependent \textit{random sampling} of the
Markov chain of the following form: It was assumed that there is positive
real-valued function $V$
on the state space $\state$ of a discrete-time Markov chain $\{x_t\}_{t\ge
  0}$, and an increasing sequence of
stopping times $\{\tau_i\}_{t\ge 0}$, with $\tau_0=0$, such that for
each~$i$,
\begin{equation}
\Expect[V(x_{\tau_{i+1}}) \mid \clF_{\tau_i}]   \le V(x_{\tau_i})  - \delta(x_{\tau_i}),
\label{e:RandomDrift}
\end{equation}
where the function $\delta\colon\state\to \Re$ is positive (bounded away from zero) outside of a ``small set'',
and $\clF_{\tau_i}$ denotes the filtration of ``events up to time $\tau_i$''. We will make this more precise later in the paper. Further relevant work include \cite{Fralix} and \cite{ConnorFort}.

Motivation for studying such problems comes from networked
control systems and communication systems: For many networked control scenarios, access to information or application of a control action in a
system is limited to random event times. As examples for such settings, there
has been significant research on stochastic stabilization of networked control
systems and information theory; as in stabilization of adaptive quantizers
studied in source coding \cite{Kieffer}, \cite{KiefferDunham} and control
theory \cite{LiberzonNesic}, \cite{BrockettLiberzon}, \cite{NairEvans}, \cite{YukTAC2010}. A specific example involving control over an erasure
channel is given in \cite{YukMeynTAC2010}, where non-zero stabilizing actions of a controller are applied to a system at certain event driven times and stochastic stability is shown using drift conditions and martingale techniques. For an extensive discussion, see
\cite{YukselBasarBook}. The methods of random-time drift criteria can also be
applied to models of networked control systems with delay-sensitive information
transmission,  for example, for studying the effects of randomness in the delay
for transmission of sensor or controller signals (see, e.g., \cite{Cloosterman}, \cite{Quevedo2}, \cite{LipsaMartins2}, \cite{WuAraposthasis}).

One other, increasingly prominent, area is {\em event-triggered
  feedback control} systems (see e.g. \cite{Event0}, \cite{QuevedoYuksel},
\cite{Event1}, \cite{Event3}, \cite{Event4}) where the event instances
constitute the stopping-times. The study of such systems is practically relevant since an event-based clock is usually more efficient than a time-triggered clock for control under information or actuation costs. The literature on such systems has primarily focused on the stabilization of such systems and we hope that the analysis in this paper will be useful for both stabilization and optimization of such systems: If the objective is to compute optimal solutions to an average cost optimization problem for an event triggered setup, a powerful approach is the {\it discounted limit approach} \cite{HernandezLermaMCP} \cite{survey}. This method typically requires geometric or sufficiently fast subgeometric convergence conditions to establish the existence of a solution to an average cost optimality equation or inequality \cite{HernandezLermaMCP}. The rate of convergence results in this paper  will be useful in such contexts. Furthermore, rates of convergence to equilibrium in Markov chains are useful in bounding the
distribution of transient events and the approximate computation of optimal costs under
ergodicity properties. In addition, as documented extensively in the literature,
Markov Chain Monte Carlo algorithms require a tedious analysis on rates of
convergence bounds to obtain probabilistically guaranteed simulation times, see,
e.g., \cite{CTCN}, \cite{Rosenthal}. Furthermore, as has been discussed in \cite{antos2008learning} and \cite{SaldiLinderYukselTAC14}, approximation methods for optimization of Markov Decision Processes benefit from the presence of sufficiently fast mixing/rates of convergence conditions. 

In this paper, we extend recent works on random-time drift analysis \cite{YukMeynTAC2010} to obtain criteria for rates of convergence under subgeometric and geometric rate functions.

The rest of the paper is organized as follows. In Section \ref{s:ss}, we provide background information on Markov chains and rates of convergence to
equilibrium. Section \ref{SectionRTSDD} contains the rate of convergence
results under random-time state-dependent drift conditions. Section
\ref{SectionE} contains an example from networked control.

\section{Markov Chains, Stochastic Stability, and Rates of Convergence}
\label{s:ss}
In this section, we review some definitions and background material
relating to Markov chains and their convergence to equilibrium.

\subsection{Preliminaries}

We let $\{x_t\}_{t\ge 0}$ denote a discrete-time Markov chain with state space
$\state$. The
basic assumptions of \cite{MeynBook} are adopted, see \cite{RamiroThesis} for a more comprehensive introduction: It is assumed that $\state$ is
a complete separable metric space; its Borel $\sigma$-field is denoted by
$\clB(\state)$.  The transition probability is denoted by $P$, so that for any
$x\in\state$, $A\in\bx$, the probability of moving in one step from the state
$x$ to the set $A$ is given by $ \Prob(x_{t+1}\in A \mid x_t=x) = P(x,A)$.  The
$n$-step transitions are obtained via composition in the usual way, $
\Prob(x_{t+n}\in A \mid x_t=x) = P^n(x,A)$, for any $n\ge1$.  The transition law
acts on measurable functions $f\colon\state\to\Re$ and measures $\mu$ on $\bx$
via  $Pf\, (x)\eqdef \int_{\state} P(x,dy) f(y),\quad x\in\state,$ and $\mu P\,
(A) \eqdef \int_{\state} \mu(dx)P(x,A)$, $A\in \clB(\state)$. A probability measure $\pi$ on $\bx$ is
called invariant if $\pi P= \pi$, i.e.,
\[
\int \pi(dx) P(x,A) = \pi(A),\qquad A\in\bx.
\]

For any initial probability measure $\nu$ on $\bx$ we can construct a stochastic
process $\{x_t\}$ with transition law $P$ satisfying $x_0\sim \nu$.  We let
$\Prob_\nu$ denote the resulting probability measure on the sample space
$(\state, \clB(\state))^{\infty}$, with the usual convention for
$\nu=\delta_{x}$ (where $\delta_x$ is the probability measure defined by $\delta_x(A)= 1_A(x)$ for all Borel $A$ and $1_E(x)$ denotes the indicator function for the event $\{x \in E\}$) when the initial state is $x\in\state$, in which case we write $\Prob_x$ for the resulting probability measure. Likewise, $E_x$ denotes the expectation operator when the initial condition is given by $x_0=x$.

When $\nu=\pi$ (the invariant measure),  the resulting process is stationary. For a set $A\in\bx$ we denote,
\begin{equation}
\tau_A\eqdef \min\{t \ge 1 :  x_t \in A\}.
\label{e:tauA}
\end{equation}
\begin{definition}
Let  $\varphi$ denote  a $\sigma$-finite measure on $\bx$. The Markov chain is
called \textit{$\varphi$-irreducible} if for any $x\in\state$, and  any
$B\in\bx$ satisfying $\varphi(B)>0$,     we have $
\Prob_{x}\{\tau_B<\infty\} >0$. A $\varphi$-irreducible Markov chain is \textit{aperiodic}  if  for any $x\in\state$, and any $B\in\bx$ satisfying $\varphi(B)>0$,   there exists $n_0=n_0(x,B)$ such that $
 P^n(x,B)>0 \qquad \hbox{\it for all \ } n\ge n_0$. A $\varphi$-irreducible Markov chain is \textit{Harris recurrent}  if   $\Prob_{x}(\tau_B < \infty  ) =1 $ for any $x\in\state$, and any $B\in\bx$ satisfying $\varphi(B)>0$.  It is \textit{positive Harris recurrent} if in addition there is an invariant probability measure $\pi$.
\end{definition}

A maximal irreducibility measure is one with respect to which all other
irreducibility measures are absolutely continuous. Define
$\clB^+(\state)=\left\{ A\in \clB(\state): \psi(A)>0\right\}$,  where $\psi$ is a maximal irreducibility measure. We refer to sets in $\clB^+(\state)$ as {\it reachable}.

A set $A \in \clB(\state)$ is full if $\psi(A^c)=0$ for a maximal irreducibility measure $\psi$. A set $A \in \clB(\state)$ is absorbing if $P(x,A)=1$ for all $x \in A$. In an irreducible Markov chain every absorbing set is {\it full}.
\begin{definition} A set $\alpha\in \clB^+(\state)$ is an atom if for all $x,y \in \alpha$ $P(x,\,\cdot\, )=P(y,\,\cdot\, )$.
\end{definition}
The concept of an atom is extremely important as it gives us a
fundamental unit, where all the points of a reachable set act together. This
allows, through the {\it cycle} equation, an invariant probability measure
$\pi(A)=E_{\alpha}\Big[\sum_{k=0}^{\tau_{\alpha}-1} 1_A(x_k)\Big] /
E_{\alpha}[\tau_{\alpha}]$. When the state space is not countable, one typically needs to
artificially construct such an atom, as we discuss further below.
\begin{definition}
A set $C \in \clB^+(\state)$  is $(n_0,\epsilon, \nu)$-small if
\[P^{n_0}(x,B)\geq \epsilon \nu(B) \quad \forall B \in \clB(\state), x\in C \]
where $n_0 \ge 1$, $\epsilon \in (0,1)$, and $\nu$ is a positive measure on
$(\state,\clB(\state))$.
\end{definition}
An important fact is that small sets exist, see Theorem 5.2.1 of
\cite{MeynBook}.
\begin{fact}\label{Small sets exist}
For an irreducible Markov chain, every set $A \in \mathcal{B^+(\state)}$ contains a small set in $\mathcal{B^+(\state)}$.
\end{fact}
\begin{definition} A set $C \in \clB^+(\state)$ is called $\kappa$-petite if
  there is a positive measure $\kappa$ on $\clB(\state)$ and a probability distribution
  $a$ on $\mathbb{Z}_+=\{0,1,2\ldots\}$  such that
\begin{equation}\label{eq: petite set def}
\sum\limits_{n=0}^{\infty}a(n)P^{n}(x,B)\geq  \kappa(B) \quad \text{for all }  B\in \clB(\state),\; x\in C.
\end{equation}
\end{definition}

The convolution of two functions $f,g: \mathbb{Z}_+ \rightarrow \mathbb{R}$,
denoted by $f \ast g$, is defined as usual by $f \ast g (n)= \sum\limits_{k=0}^{n}f(k)g(n-k)$, for all $n \in \mathbb{Z}_+$. The next lemma follows from Lemma 5.5.2 of \cite{MeynBook} and allows us to assume without loss of generality that for an irreducible Markov chain, if a set is $\kappa$-petite, then $\kappa$ can be replaced by maximal irreducibility measure (or equivalently $\kappa$ can be assumed maximal).
\begin{lemma}\label{lemma: petite irr. meas.}
  If an irreducible Markov chain has some set $C \in \clB^+(\state)$ that is
  $\kappa$-petite for some distribution $a$, then $C$ is
  $\psi$-petite for the distribution $a \ast f (n)$ where $f(n)=2^{-n-1}$ and
  $\psi$ is a maximal irreducibility measure.
\end{lemma}


An important result is the equivalence of small sets and petite sets.
\begin{theorem}[\cite{MeynBook}, Theorem 5.5.3] \label{lemma: petite is small}
For an aperiodic and irreducible Markov chain every petite set is small.
\end{theorem}

Small sets are analogous to compact sets in the stability theory for
$\varphi$-irreducible Markov chains.  In most applications of
$\varphi$-irreducible Markov chains we find that any compact set is small -- in
this case, $\{x_t\}$ is called a \emph{T-chain} \cite{MeynBook}. The
equivalence of small sets and petite sets can be used cleverly to show that all
petite sets are petite for some distribution that has finite mean. The next
theorem follows from Propositions 5.5.5 and 5.5.6 of \cite{MeynBook}.

\begin{theorem}\label{theorem: petite set finite mean}
For an aperiodic and irreducible Markov chain every petite set is petite with a maximal irreducibility measure for a distribution with finite mean.
\end{theorem}

Invoking (\ref{eq: petite set def}), we will use Theorem \ref{theorem: petite set finite mean} repeatedly with a set $C$ that is $\kappa$-petite for some distribution $a(\,\cdot\,)$ to achieve bounds on hitting times for any $B \in \clB^+(\state)$,

\begin{eqnarray} \label{supportingPetite}
E_x\bigg[\sum\limits_{k=0}^{\tau_B-1}1_C(x_k)\bigg] &\leq&
\frac{1}{\kappa(B)}E_x\bigg[\sum\limits_{k=0}^{\tau_B-1}\sum_{n=1}^{\infty} 1_B(x_{k+n})a(n)\bigg] \nonumber \\
&\leq&
\frac{1}{\kappa(B)}\sum_{n=0}^{\infty} n a(n) =: c(B)<\infty.
\end{eqnarray}

\subsection{Regularity and Ergodicity}\label{defSec: reg and ergo}
Regularity and ergodicity are concepts closely related through the work of Meyn and Tweedie \cite{MeynBook}, \cite{MeynStateDrift} and Tuominen and Tweedie \cite{Tuominen-and-Tweedie}. The definitions below are in terms of functions $f:\state\rightarrow [1,\infty)$ and $r:\mathbb{Z}_+\rightarrow (0,\infty)$.
\begin{definition} A set $A \in \clB(\state)$ is called $(f,r)$-regular if
\[
\sup_{x\in A}E_x\bigg[\sum_{k=0}^{\tau_B-1}r(k)f(x_k)\bigg]<\infty
\]
for all $B \in \clB^+(\state)$. A finite measure $\nu$ on $\clB(\state)$ is called $(f,r)$-regular if
\[
E_{\nu}\bigg[\sum_{k=0}^{\tau_B-1}r(k)f(x_k)\bigg]<\infty
\]
for all $B \in \clB^+(\state)$, and a point $x$ is called $(f,r)$-regular if the measure $\delta_x$ is $(f,r)$-regular.
\end{definition}


To make sense of ergodicity we first need to define the $f$-norm, denoted $\lVert . \rVert_f$.

\begin{definition} For a function $f: \state \rightarrow [1,\infty)$ the $f$-norm of a measure $\mu$ defined on $(\state,\clB(\state))$ is given by
\[
\lVert \mu \rVert_f =\sup_{g \leq f}  \ \bigg| \int \mu(dx)g(x) \bigg|
\]
where the supremum  is taken over all measurable $g$ such that $g(x)\le f(x)$ for
all $x$.
\end{definition}
The commonly used total variation norm, or $TV$-norm, is the $f$-norm when
$f=1$, and is denoted by $\lVert \,\cdot\, \rVert_{TV}$.

\begin{definition} A Markov chain $\{x_t\}$ with invariant distribution $\pi$ is
  $(f,r)$-ergodic if
\begin{equation}\label{eq: (f,r) ergodic}
r(n)\lVert P^n(x,\,\cdot\, )-\pi(\,\cdot\, )\rVert_f \rightarrow 0  \quad \text{as $n \rightarrow \infty$ for all $x\in \state.$}
\end{equation}
If (\ref{eq: (f,r) ergodic}) is satisfied for a geometric $r$ (so that $r(n) = M \zeta^n$ for some $\zeta > 1$, $M < \infty$) and $f=1$ then the
Markov chain $\{x_t\}$ is called  geometrically ergodic.
\end{definition}

\subsection{The Splitting Technique and the Coupling Inequality}\label{Splitting}
Nummelin's splitting technique \cite{Nummelin} (see also \cite{AthreyaNey}) is a widely used method in the
study of Markov chains; see, e.g., Chapter 5 in \cite{MeynBook}, Proposition 3.7 and Theorem 4.1 in  \cite{Tuominen-and-Tweedie}, Section~4.2 in \cite{Rosenthal}. With an irreducible, aperiodic Markov chain $\{x_t\}$ on
state space $\state $ with transition probability $P$ and a $(m,\delta,\nu
)$-small set $C$ with finite return time, we construct an atom in order to
construct an invariant distribution for the chain.

We first review the splitting technique for the case $m=1$ (i.e. $C$ is a $(1,
\delta, \nu$-small set). Construct a new Markov chain $\{z_t\}$ on $\state
\times \{0,1\}$ by letting $z_t=(x_t,a_t)$, where $\{a_t\}$ is a sequence of
random variables on $\{0,1\}$, independent of $\{x_t\}$, except  when $x_t \in C$.

1. If $x_t \notin C$ then $x_{t+1} \sim P(x_t, \,\cdot\,)$

2. If $ x_t \in C$, then

\qquad with probability $\delta : a_t = 1$ and $ x_{t+1} \sim \nu(\,\cdot\,)$

\qquad with probability $(1-\delta) : a_t=0$ and $\displaystyle x_{t+1} \sim \frac{P(x_t,\,\cdot\,)-\delta \nu(\,\cdot\,)}{1-\delta}$.

Thus the distribution of $x_{t+1}$ given $z_t$ is
\begin{align*}
P\big(x_{t+1} \in B \,\big|\, z_t=(x_t,a_t) \in C \times \{1\}\big)& =\nu(B)\\
P\big(x_{t+1} \in B \,\big|\, z_t=(x_t,a_t) \in C \times \{0\}\big)& =\frac{P(x_t,B)-\delta
  \nu(B)}{1-\delta}
\end{align*}
Note that $\frac{P(x_t,\,\cdot\,)-\delta \nu(\,\cdot\,)}{1-\delta} \geq 0 $ is a
valid probability measure since $C$ is $(1,\delta, \nu)$-small. If $x_t \in C$,
then
\[x_{t+1} \sim \delta \nu(\,\cdot\,) + (1-\delta) \frac{P(x_t,\,\cdot\,)-\delta \nu(\,\cdot\,)}{1-\delta} = P(x_t, \,\cdot\,)\]
so the one-step transition probabilities are unchanged for $\{x_t\}$.

This construction allows one to define $S=C\times\{1\}$ as an atom
for $\{z_t\}$, and to construct an invariant distribution for $\{x_t\}$ using
$\{z_t\}$.

We specified the technique for the one step transition probability, but the same
construction applies for $(m, \epsilon, \nu)$-small sets where $m>1$
with the only change  that the $m-1$ steps after hitting $C$ at $x_t$ are
distributed conditionally on $x_t$ and $x_{t+m}$ (see Section 4.2 of
\cite{Rosenthal}). When  $m >1$,  the Markov chain $\{z_t\} $ does not
have an atom; instead it has an "$m$-step atom" in the sense that
$P^m\big((x,1),\,\cdot\, \big)=P^m\big((y,1),\,\cdot\, \big)$ for all $x,y \in C$.

A useful method to obtain bounds of convergence is through the {\it coupling inequality}. The coupling inequality bounds the total variation
distance between the distributions of two random variables by the probability they are different. Let $X,\ Y$ be two jointly distributed random variables. The following is the well known coupling inequality.
\begin{align*}
 \lVert P(X \in \,\cdot\,)-P(Y \in \,\cdot\,) {\rVert}_{TV} \leq P(X \neq Y).
\end{align*}
This inequality is useful in discussions of ergodicity when used in
conjunction with parallel Markov chains, as in Theorem 4.1 of \cite{Rosenthal},
and Theorem 4.2 of \cite{Hairer}: One tries to create two Markov chains, $x_t$
and $x'_t$, having the same one-step transition probability distribution but
driven independently until they are coupled on a small set with some fixed
probability whenever they visit the small set. Here, $x'_t$ is a stationary Markov chain. By the Coupling Inequality and
the previous discussion with Nummelin's splitting technique we have $\lVert P^n
(x,\,\cdot\, )-\pi(\,\cdot\, ) \rVert_{TV} \leq {P(x_n \neq x'_n)}$, where $x'_n
\sim \pi P^n =\pi$.

\subsection{Drift criteria for positivity}
\label{driftCriteria}

We now consider specific   formulations of the random-time drift criterion
\eqref{e:RandomDrift}.   Throughout the paper the sequence of stopping times
$\{\tau_i\}_{i\ge 0}$ is assumed to be non-decreasing, with $\tau_0=0$.

\Theorem{thm5} is the general result of \cite{YukMeynTAC2010}, providing a
single criterion for positive Harris recurrence, as well as finite ``moments''
(the steady-state mean of the function $f$ appearing in the drift condition
\eqref{e:thm5delta}).  The drift condition \eqref{e:thm5delta} is a refinement
of
\eqref{e:RandomDrift}. 

\begin{theorem}\cite{YukMeynTAC2010} \label{thm5}
Suppose that $\{x_t\}$ is a $\varphi$-irreducible and aperiodic Markov chain.   Suppose moreover that there are functions
 $V \colon \state \to (0,\infty)$,
 $\delta \colon \state \to [1,\infty)$,
 $f\colon \state \to [1,\infty)$,
a small set $C$ on which $V$ is bounded, and a constant $b \in \Re$,   such that
\begin{equation}
\begin{aligned}
\Expect[V(x_{\tau_{i+1}}) \mid \clF_{\tau_i }]  &\leq  V(x_{\tau_{i}}) -\delta(x_{\tau_i }) + b1_C(x_{\tau_i})
\\
 \Expect \biggl[\,\sum_{k=\tau_i}^{\tau_{i+1}-1} f(x_k) \, \Big|\, \clF_{\tau_i }\biggr]  &\le \delta(x_{\tau_i}). \qquad \qquad \qquad \qquad i\ge 0.
\end{aligned}
\label{e:thm5delta}
\end{equation}
Then the following hold:
\begin{romannum}
\item
$\{x_t\}$ is positive Harris recurrent, with unique invariant distribution $\pi$
\item $\pi(f)\eqdef \int f(x)\, \pi(dx) <\infty$.
\item  For any function $g$ that is bounded by $f$, in the sense that $\sup_{x}
  |g(x)|/f(x)<\infty$, we have convergence of moments in the mean, and the strong law
  of large numbers holds:
\[
\begin{aligned}
\lim_{t\to\infty} \Expect_{x}[g(x_t)] &= \pi(g)
\\
\lim_{N\to\infty} \frac{1}{N} \sum_{t=0}^{N-1} g(x_t) &= \pi(g)\qquad a.s.\,, \ x\in\state
\end{aligned}
\]
\end{romannum}
\end{theorem}


\subsection{Rates of Convergence: Geometric Ergodicity}\label{Sec: Geometric Ergodicity}
In this section, following \cite{MeynBook} and \cite{Rosenthal}, we review
results stating that
a strong type of ergodicity, geometric ergodicity, follows from a simple drift
condition. An irreducible Markov chain is said to satisfy the {\it univariate
  drift condition} if there are constants $\lambda \in (0,1)$ and $b < \infty$,
along with a function $V:\state \rightarrow [1,\infty)$, and a small set $C$
such that
\begin{eqnarray}\label{univariate drift}
PV \leq \lambda V + b 1_C .
\end{eqnarray}

Using the coupling inequality, Roberts and Rosenthal \cite{Rosenthal} prove  that
geometric ergodicity follows from the univariate drift condition. We also note that the univariate drift condition allows us to
assume that $V$ is bounded on $C$ without any loss (see Lemma 14 of
\cite{Rosenthal}).

\begin{theorem}[\protect{\cite[Theorem 9]{Rosenthal}}]
  Suppose $\{x_t\}$ is an aperiodic, irreducible Markov chain with invariant
  distribution $\pi$. Suppose $C$ is a $(1,\epsilon,\nu)$-small set
  and $V:\state  \rightarrow [1,\infty)$ satisfies the univariate drift condition with
  constants $\lambda \in (0,1)$ and $b < \infty$. Then $\{x_t\}$ is geometrically ergodic.
\end{theorem}

That geometric ergodicity follows from the univariate drift condition with a
small set $C$ is proven by Roberts and Rosenthal by using the coupling
inequality to bound the $TV$-norm, but an alternate proof is given by Meyn and
Tweedie \cite{MeynBook} resulting in the following theorem.

\begin{theorem}[\protect{\cite[Theorem 15.0.1]{MeynBook}}]\label{theorem: geometric ergodicity MT}
  Suppose  $\{x_t\}$ is an aperiodic and irreducible Markov chain. Then the
  following are equivalent:

(i) $E_x[\tau_B]<\infty$ for all $x \in \state$, $B \in \clB^+(\state)$, the invariant distribution $\pi$ of $\{x_t\}$ exists and there exists a petite set $C$, constants $\gamma<1$, $M>0$ such that for all $x \in C$
\[\lvert P(x,C)-\pi(C) \rvert < M\gamma^n .\]
(ii) For a petite set $C$ and for some $\kappa > 1$
\[\sup_{x\in C}E_x[\kappa^{\tau_C}]<\infty.\]
(iii) For a petite set $C$, constants $b >0$ $\lambda \in (0,1)$, and a function $V:\state \rightarrow [1,\infty]$ (finite for some $x$) such that
\[PV \leq \lambda V+b1_C.\]
Any of the conditions imply that there exists $r>1$, $R<\infty$ such that for
any $x $
\[\sum_{n=0}^{\infty}  r^n \lVert P^n(x,\,\cdot\,)-\pi(\,\cdot\,)\rVert_V \leq RV(x). \]
\end{theorem}
We note for future reference that if (iii) above holds, (ii) holds for for all $\kappa \in (1,\lambda^{-1})$.

\subsection{Rates of Convergence: Subgeometric Ergodicity}\label{Sec: Subgeometric Ergodicity}
%

Here, we review the class of subgeometric rate functions (see Section 4 in
\cite{Hairer}, Section 5 in \cite{ConnorFort}, and \cite{MeynStateDrift},
\cite{MeynBook}, \cite{Douc}, \cite{Tuominen-and-Tweedie}). Let $\Lambda_0$ be
the family of functions $r:\mathbb{Z}_+\rightarrow [0,\infty)$ satisfying
\[ \text{$r$ is non-decreasing,} \quad r(1)\geq 2\]
and
\[\frac{\log r(n)}{n}\downarrow 0 \quad \text{as $n \rightarrow \infty$}. \]
The second condition implies that for all $r\in \Lambda_0$
\begin{equation}\label{eq:r dynamics}
r(m+n) \leq r(m)r(n) \quad\text{for all $m,n \in \mathbb{Z}_+.$}
\end{equation}

The class of subgeometric rate functions $\Lambda$ defined in
\cite{Tuominen-and-Tweedie} is the class of sequences $r$ for which there exists
a function $r_0\in \Lambda_0$ such that \[0 < \liminf_{n \to \infty}
\frac{r(n)}{r_0(n)} \leq \limsup_{n \to \infty} \frac{r(n)}{r_0(n)} < \infty. \]


The main theorems we use to construct conditions on subgeometric rates of
convergence are due to Tuominen and Tweedie \cite{Tuominen-and-Tweedie}.

\begin{theorem}[\cite{Tuominen-and-Tweedie}, Theorem 2.1]
  Suppose $\{x_t\}$ is an irreducible and aperiodic Markov chain with state
  space $\state$ and transition probability $P$. Let
  $f:\state\rightarrow[1,\infty)$ and $r \in \Lambda$ be given. The following
  are equivalent:

\noindent(i) There exists a petite set $C \in \mathcal{B}(\state)$ such that
\[
\sup_{x\in C} E_x \bigg[\sum_{k=0}^{\tau_C -1}r(k)f(x_k)\bigg]<\infty .
\]

\noindent(ii) There exist a sequence $\{V_n\}$ of functions
$V_n:\state\rightarrow[0,\infty]$, a petite set $C \in \mathcal{B}(\state)$, and
$b>0$ such that $V_0$ is bounded on $C$,
\[
V_0(x)=\infty \text{ implies } V_1(x)=\infty,
\]
and
\begin{eqnarray}
PV_{n+1}\leq V_n-r(n)f+br(n)1_{C}, \quad n\in \mathbb{Z}_+. \label{TTDriftCond}
\end{eqnarray}

\noindent(iii) There exists an $(f,r)$-regular set $A\in \mathcal{B^+(\state)}$.

\noindent(iv) There exists a full absorbing set $S$ which can be covered by a
countable number of $(f,r)$-regular sets.

\label{Main Theorem Tweedie}
\end{theorem}

\begin{theorem}[\protect{\cite[Theorem
    4.1]{Tuominen-and-Tweedie}}]\label{theorem: ergodicity}
Suppose an aperiodic and irreducible Markov chain $\{x_t\}$ satisfies the
equivalent conditions  (i)-(iv)
of Theorem \ref{Main Theorem Tweedie} with $f: \state\rightarrow [1,\infty)$ and
$r \in \Lambda$. Then the Markov chain is $(f,r)$-ergodic, i.e.,
\[
\lim_{n\to \infty} r(n)\lVert P^n(x,\,\cdot\, )-\pi \rVert_f =0.
\]
\end{theorem}

The proof of this result relies on a first-entrance last-exit decomposition \cite{MeynBook} of the transition probabilities; see Section 13.2.3 of \cite{MeynBook}.

The conditions of Theorem \ref{Main Theorem Tweedie} may be hard to check,
especially (ii), comparing a sequence of Lyapunov functions $\{V_k\}$ at each time
step. We briefly discuss the methods of Douc et al. \cite{Douc} (see also Hairer
\cite{Hairer}) that extend the subgeometric ergodicity results and show how to construct subgeometric rates of ergodicity from a
simpler drift condition. \cite{Douc} assumes that there exists a function $V:
\state\rightarrow [1, \infty]$, a concave monotone non-decreasing differentiable
function $\phi:[1,\infty]\rightarrow(0,\infty]$, a set $C \in \clB(\state)$ and
a constant $b \in \mathbb{R}$ such that
\begin{equation}
  \label{eq_drift}
PV+\phi {\circ} V \leq V +b1_C.
\end{equation}

If an aperiodic and irreducible Markov chain $\{x_t\}$ satisfies the above with
a petite set $C$, and if $V(x_0) < \infty$, then it can be shown that $\{x_t\}$
satisfies Theorem \ref{Main Theorem Tweedie}(ii). Therefore $\{x_t\}$ has
invariant distribution $\pi$ and is $(\phi {\circ} V, 1)$-ergodic so that
$\lim\limits_{n\rightarrow \infty}\lVert P^n(x,\,\cdot\, )-\pi(\,\cdot\,
)\rVert_{\phi {\circ} V}=0$ for all $x$ in the set $\{x:V(x)< \infty\}$  of
$\pi$-measure 1. The results by Douc et al.\  build then on trading off
$(\phi \circ V,1)$ ergodicity for $(1,r_\phi)$-ergodicity for some rate function
$r_\phi$, by carefully constructing the function utilizing
concavity; see Propositions 2.1 and 2.5 of \cite{Douc} and Theorem 4.1(3) of
\cite{Hairer}.

To achieve ergodicity with a nontrivial rate and norm one can invoke a result
involving the class of {\it pairs of ultimately non-decreasing functions},
defined in \cite{Douc}. The class $\mathcal{Y}$ of  pairs of ultimately non-decreasing
  functions consists of  pairs
$\Psi_1,\Psi_2:\state\rightarrow[1,\infty)$ such that $\Psi_1(x)\Psi_2(y)\leq
x+y$ and $\Psi_i(x) \rightarrow \infty$ for one of $i=1,2$.

\begin{proposition}\label{prop: Y ergo ext}
Suppose $\{x_t\}$ is an aperiodic and irreducible Markov chain that is both
$(1,r)$-ergodic and $(f,1)$-ergodic for some $r \in \Lambda$ and
$f:\state\rightarrow [1,\infty)$. Suppose $\Psi_1,\Psi_2:\state\rightarrow
[1,\infty)$ are a pair of ultimately non-decreasing functions. Then $\{x_t\}$ is
$(\Psi_1\circ f, \Psi_2 \circ r)$-ergodic.
\end{proposition}

Therefore we can show that if  $(\Psi_1,\Psi_2)\in \mathcal{Y}$ and  a Markov
chain satisfies the condition \eqref{eq_drift}, then it is $(\Psi_1 \circ
\phi\circ V, \Psi_2 \circ
r_\phi)$-ergodic.

\section{Rates of Convergence under Random-Time State-Dependent Drift}\label{SectionRTSDD}

The second condition of Theorem \ref{Main Theorem Tweedie} assumes that a
deterministic sequence of functions $\{V_n \}$ exists and satisfies the drift condition (\ref{TTDriftCond}).

We apply Theorem \ref{Main Theorem Tweedie} to the case where the
Foster-Lyapunov drift condition holds not for every $n$ but for a sequence of
stopping times $\{\tau_n\}$. Our goal is to reveal a relation between the
stopping times $\{\tau_n\}$ where a drift condition holds and the rate function
$r$, so that we obtain $(f,r)$-ergodicity.

\subsection{A general result on ergodicity }

The following result builds on and generalizes Theorem 2.1 in \cite{YukMeynTAC2010}.

\begin{theorem}\label{Prop: random time app}
Let $\{x_t\}$ be an aperiodic and irreducible Markov chain with a small set $C$. Suppose there are functions $V: \state \rightarrow (0,\infty)$ with $V$ bounded on $C$, $f:\state \rightarrow [1,\infty), \delta :\state \rightarrow [1,\infty)$, a constant $b \in \mathbb{R}$, and $r \in \Lambda$ such that for a sequence of stopping times $\{\tau_n\}$
\begin{eqnarray}
\label{eq: random drfit}E[V(x_{\tau_{n+1}})\mid x_{\tau_n}]\leq V(x_{\tau_n})-\delta(x_{\tau_n}) +b1_C(x_{\tau_n}) \nonumber\\
\label{eq:random sum} E\bigg[\sum\limits_{k=\tau_n}^{\tau_{n+1}-1}f(x_k)r(k)\,
\Big| \,
\clF_{\tau_n}\bigg] \leq \delta(x_{\tau_n}).
\end{eqnarray}
Then $\{x_t\}$ satisfies Theorem \ref{Main Theorem Tweedie} and is
$(f,r)$-ergodic.
\end{theorem}

\begin{proof}
The proof is similar to the proof of the Comparison Theorem of \cite{MeynBook} as well as Theorem 2.1(i) in \cite{YukMeynTAC2010}. We may assume $r\in \Lambda_0$. We define sampled hitting
  times $\gamma_B =\min\{n>0:x_{\tau_n} \in B\}$ for all $B \in
  \clB^+(\state)$ and $\gamma_B^N=\min(N,\gamma_B)$. Since $\{x_{\tau_n}\}$ satisfies the drift condition, it
  follows that for $x \in C$
\begin{align*}
  E_x\bigg[\sum\limits_{n=0}^{\gamma^N_C-1}\delta(x_{\tau_n})\bigg] \leq &
  V(x)+bE_x\bigg[\sum\limits_{n=0}^{\gamma^N_C-1}1_C(x_{\tau_n})\bigg] \leq V(x)+b
\end{align*}
which is finite since $V$ is bounded on $C$ by assumption. An application of the monotone convergence theorem then gives
\begin{align*}
  E_x\bigg[\sum\limits_{n=0}^{\gamma_C-1}\delta(x_{\tau_n})\bigg] \leq &
  V(x)+bE_x\bigg[\sum\limits_{n=0}^{\gamma_C-1}1_C(x_{\tau_n})\bigg] \leq V(x)+b
\end{align*}

Since $\tau_B \leq \tau_{\gamma_B}$ for all $B \in \clB^+(\state)$ by
definition, we have
\[
E_x\bigg[\sum\limits_{n=0}^{\tau_C-1}f(x_n)r(n)\bigg]\leq
E_x\bigg[\sum\limits_{n=0}^{\gamma_C-1}\delta(x_{\tau_n})\bigg] \leq V(x)+b
\]
so  $C$
is a petite set which  satisfies
\[
\sup_{x\in
  C}E_x\bigg[\sum\limits_{n=0}^{\tau_C-1}r(n)f(x_n)\bigg]\leq \sup_{x\in C}V(x)+b <
\infty.
\]
This  means that the Markov chain $\{x_n\}$ satisfies Theorem \ref{Main
  Theorem Tweedie}(i) and is $(f,r)$-ergodic.
\end{proof}

\subsection{On petite sets and sampling}

Unfortunately the techniques we reviewed earlier that rely on petite sets (specifically
Theorem \ref{theorem: petite set finite mean}) become unavailable in the random
time drift setting as a petite set $C$ for $\{x_n\}$ is not necessarily petite
for $\{x_{\tau_n}\}$. To be able to relax conditions on the behavior of $V$ on
$C$,  we can place one of the following two conditions on the stopping times or
require that $V$ is bounded on $C$.

For an analogous application of Theorem \ref{theorem: petite set finite mean} in
the random time setting we define sampled hitting times for any $B \in
\clB^+(\state)$ as $\gamma_B=\min\{n>0: x_{\tau_n} \in B\}$.

\begin{lemma}\label{ch5: random sampling preserves smallness}
  Suppose $\{x_t\}$ is an aperiodic and irreducible Markov chain. If there
  exists sequence of stopping times $\{\tau_n\}$ independent of $\{x_t\}$, then
  any $C$ that is small for $\{x_t\}$ is petite for $\{x_{\tau_n}\}$.
\end{lemma}

\begin{proof}
  Since $C$ is petite, it is small by Theorem \ref{lemma: petite is small} for
  some $m$.  Let $C$ be $(m, \delta, \nu)$-small for $\{x_t\}$.
\begin{eqnarray}
&&  P^{\tau_{1}}(x,\,\cdot\,)= \sum\limits_{k=1}^{\infty}P(\tau_{1}=k)P^k(x,\,\cdot\,) \nonumber \\
&&  \geq  \sum\limits_{k=m}^{\infty}P(\tau_{1}=k)\int P^m(x,dy)P^{k-m}(y,\,\cdot\,) \nonumber\\
&& \geq \sum\limits_{k=m}^{\infty} P(\tau_{1}=k) \int 1_C(x) \delta \nu(dy)  P^{k-m}(y,\,\cdot\,)
\end{eqnarray}
which is a well defined measure. Therefore defining $\kappa(\,\cdot\,)=\int
\nu(dy) \sum\limits_{k=m}^{\infty} P(\tau_{1}=k) P^{k-m}(y, \,\cdot\,)$, we have
that $C$ is $(1,\delta,\kappa)$-small for $\{x_{\tau_n}\}$.
\end{proof}
The above allows us to uniformly bound
$E_x\big[\sum_{n=0}^{\gamma_B-1}1_C(x_{\tau_n})\big]$ when the stopping times are
independent of the Markov chain, by an application of Theorem \ref{theorem:
  petite set finite mean} and (\ref{supportingPetite}).

The independence of stopping times $\{\tau_n\}$ of $\{x_t\}$ is a restrictive
condition that {\it event triggered systems} cannot satisfy since in such systems the stopping times  depend explicitly on the state process
hitting certain sets. One useful example where independence of stopping times
can be enforced is given in \cite{Quevedo2} where a system controlled over an
unreliable network is affected by variable transmission delays between the controller and the plant.

For the event-triggered case we will derive a useful result which will be used to show that in the drift equations of the form (\ref{eq: random drfit}), the Lyapunov function $V$ may not need to be assumed bounded on $C$.

The proof of the next result follows directly from the definition of $\tau_n$.
\begin{lemma}\label{ch5: event-triggered case}
  Suppose $\{\tau_n\}$ are the subsequent hitting times of a sequence of sets
  $\{E_n\}$ in $\clB^+(\state)$, so that $\tau_{n+1}=\min\{t>\tau_n:x_t \in
  E_{n+1}\}$.  If $\bigcap_{n=0}^{\infty} E_n \in \clB^+(\state)$ then for any
  reachable $B \subset \bigcap_n E_n$, we have $\tau_{\gamma_B}=\tau_B$.
\end{lemma}

\begin{assumption} \label{completeAssumption}
The stopping times are as in Lemma \ref{ch5: event-triggered case} and $C \subset \bigcap_{n=0}^{\infty} E_n$.
\end{assumption}

Recall that by Theorem \ref{theorem: petite set
  finite mean} a petite set $C$ is petite with a maximal irreducibility measure
$\kappa$ for a distribution $a$ with finite mean, so for any $B \in
\clB^+(\state)$ we have that $\sum \limits_{n=0}^{\infty}a(n)P^{n}(x,B)
\geq \kappa(B) 1_C(x)$. Assumption \ref{completeAssumption} then implies that if any $C$ is petite for $\{x_t\}$, then for some $\tilde{C} \subset C \subset \bigcap_{n=0}^{\infty} E_n$, we have that
\begin{eqnarray}
&& E_x\bigg[\sum\limits_{k=0}^{\gamma_{\tilde{C}}-1} 1_C(x_{\tau_k})\bigg] \nonumber \\
&& \leq E_x\bigg[\sum\limits_{k=0}^{\tau_{\tilde{C}}-1} 1_C(x_k)\bigg] \nonumber \\
&& \leq \frac{1}{\kappa(\tilde{C})}E_x\bigg[\sum\limits_{k=0}^{\tau_{\tilde{C}}-1}\sum_n 1_{\tilde{C}}(x_{k+n})a(n)\bigg] \nonumber \\
&& = \frac{1}{\kappa(\tilde{C})} \sum_n a(n) E_x\bigg[\sum\limits_{k=0}^{\tau_{\tilde{C}}-1} 1_{\tilde{C}}(x_{k+n})\bigg] \nonumber \\
&& \leq \frac{1}{\kappa(\tilde{C})} \sum_n a(n) n = c(\tilde{C})<\infty.
\end{eqnarray}

Hence if the stopping times satisfy the conditions in
   Lemma \ref{ch5: random sampling preserves smallness} or Lemma~\ref{ch5: event-triggered
    case}, we can drop the condition that $V$
  is bounded on $C$, by applying  Chap. 11 of \cite{MeynBook} and Proposition 5.5.6
  of \cite{MeynBook} to $\{x_{\tau_n}\}$ and noting that  by (\ref{eq:random
    sum}),
  $\{x_t\}$ satisfies Theorem \ref{Main Theorem Tweedie}(i). This
  follows since the drift condition implies for any $B \in \clB^+(\state)$,
  $E_x\big[\sum\limits_{n=0}^{\gamma_B-1}\delta(x_{\tau_n})\big] \leq
  V(x)+bE_x\big[\sum\limits_{n=0}^{\gamma_B-1}1_C(x_{\tau_n})\big]$,  where the last term
  is bounded if the conditions of either  Lemma \ref{ch5: random sampling
    preserves smallness} or Lemma~\ref{ch5: event-triggered case} and Assumption \ref{completeAssumption} are
  satisfied.

It is interesting to note that the two extreme cases of the stopping times,
either independent of or completely dependent on the Markov chain, both give
similarly useful relaxations.

\subsection{Subgeometric ergodicity}

The second inequality (\ref{eq:random sum}) may be hard to check as it does not
provide means for checking the relation between the stopping times $\{\tau_n\}$
and the rate function $r$ since the function depends on $k$ in a non-explicit fashion. In the following, the relationship of the criteria with the rate function $r$ is relative to the stopping time.

We assume that $r\in \Lambda_0$ and thus $r$ satisfies $r(m+n)\leq r(m)r(n)$.  

\begin{theorem}\label{ch5: r ergo with uni drift}
  Let $\{x_t\}$ be an aperiodic and irreducible Markov chain with a small set
  $C$. Suppose there exist $V: \state \rightarrow [1,\infty)$ which is
  bounded on $C$ and for some $\epsilon>0$, $\lambda \in (0,1)$, $\lambda V(x) \leq V(x)-\epsilon$ for all $x\notin C$, and $b \in \mathbb{R}$ such that for an increasing sequence of stopping times $\{\tau_n\}$
\begin{eqnarray}\label{DriftCond1}
E\big[V(x_{\tau_{n+1}}) \mid {\cal F}_{\tau_n}] \leq \lambda{V(x_{\tau_n})} +b1_C(x_{\tau_n}).
\end{eqnarray}
If
\begin{eqnarray}\label{DriftCond00}
\sup_k E\big[\sum_{n=\tau_k}^{\tau_{k+1}-1}r(n-\tau_k)\mid \clF_{\tau_k}\big] =: M <
\infty
\end{eqnarray}
and
\begin{eqnarray}\label{DriftCond01}
\sup_k E\big[r(\tau_{k+1}-\tau_k)\mid \clF_{\tau_k}\big]\leq
\lambda^{-1},
\end{eqnarray}
then $\{x_t\}$ satisfies Theorem \ref{Main Theorem Tweedie} with $f=1$ and is
$(1,r)$-ergodic.
\end{theorem}

\begin{proof}
Suppose that instead of (\ref{DriftCond1}), we have that
\begin{align}
E[V(x_{n+1})\mid {\cal F}_n]\leq& \lambda V(x_n) +b1_C(x_n). \label{criterionDrift2}
\end{align}

It follows then that the sequence $\{M_n\}$ defined by
  \[M_n= \lambda^{-n} V(x_n) - \sum\limits_{k=0}^{n-1} b1_C(x_k) \lambda^{-(k+1)} ,\]
  with $M_0 = V(x_0)$, is a supermartingale. Then, with (\ref{DriftCond1}), defining
  $\gamma_B^N=\min\{N,\gamma_B\}$ for $B \in \clB^+(\state)$ gives, by Doob's optional sampling theorem,
\begin{eqnarray}
&& E_x\bigg[ \lambda^{-\gamma_B^N} V(x_{\tau_{\gamma_B^N}})\bigg]\leq
V(x) \nonumber \\
&& \quad \quad \quad \quad +E_x\bigg[\sum\limits_{n=0}^{\gamma_B^N-1}b1_C(x_{\tau_n}) \lambda^{-(n+1)} \bigg]
\end{eqnarray}
for any $B \in \clB^+(\state)$, and $N\in \mathbb{Z}_+$.

Since $V$ is bounded above on $C$, we have that $C \subset \{ V \leq L_1\}$ for some $L_1$ and thus,
\[\sup_{x\in
  C}E_x\bigg[ \lambda^{-\gamma_C^N} V(x_{\tau_{\gamma_C^N}})\bigg]\leq
L_1 + \lambda^{-1} b\]
and by the monotone convergence theorem, and the fact that $V$ is bounded from below by 1 everywhere and bounded from above on $C$,
\[\sup_{x\in
  C}E_x\bigg[  \lambda^{-\gamma_C} V(x_{\tau_{\gamma_C}})\bigg]\leq L_1(L_1 +\lambda^{-1}b).\]

Now, for any $r \in \Lambda$, we have
\begin{eqnarray}
&& \sup_{x\in C}E_x\bigg[\sum\limits_{n=0}^{\tau_C-1}r(n)\bigg] \leq \sup_{x\in
    C}E_x\bigg[\sum\limits_{n=0}^{\tau_{\gamma_C}-1}r(n)\bigg], \nonumber 
\end{eqnarray}
and since $r(m+n)\leq r(m)r(n)$ by (\ref{eq:r dynamics}), we obtain through iterated expectations that
\begin{eqnarray}
&&
  \sup_{x\in C}E_x\bigg[\sum\limits_{n=0}^{\tau_C-1}r(n)\bigg] \nonumber \\
&&   \leq\sup_{x\in
    C}E_x\bigg[\sum\limits_{k=0}^{{\gamma_C-1}} E\bigg[ E\Big[\sum\limits_{n=\tau_k}^{\tau_{k+1}-1}r(n-\tau_k)
  \, \Big|\,  \clF_{\tau_k}\Big]  \nonumber \\
&& \quad \quad \times \prod\limits_{m=1}^{k}  r(\tau_{m}-\tau_{m-1}) \bigg] \bigg]. \nonumber
\end{eqnarray}

Now, with (\ref{DriftCond00})-(\ref{DriftCond01}), and by the fact that $V$ is bounded from below by $1$, it follows that
\begin{eqnarray}
&& \sup_{x\in C}E_x[\sum\limits_{n=0}^{\tau_{\gamma_C}-1}r(n)]  \nonumber \\
&& \leq \sup_{x\in
  C}E_x[\sum\limits_{n=1}^{{\gamma_C}-1}E_x[MV(x_{\tau_{n}}) \lambda^{-n}]] \nonumber \\
&& \leq L_1 \sup_{x\in C}E_x[\sum\limits_{n=0}^{{\gamma_C}-1}M(V(x)+\lambda^{-1} b)]
\end{eqnarray}
so that
\[
\sup_{x\in C}E_x\bigg[\sum\limits_{n=0}^{\tau_{\gamma_C}-1}r(n)\bigg]\leq M L_1 (L_1+\lambda^{-1} b) \sup_{x\in
  C}E_x[{\gamma_C}].
\]
From (\ref{DriftCond1}) and the  condition $\lambda V(x)\leq V(x) - \epsilon$ for $x \notin C$, we get that for all $x \in C$
\begin{eqnarray}
&& E_x\big[V(x_{\tau_{\gamma_C}})\big] \nonumber \\
&& \leq \bigg( V(x)-E_x\bigg[\sum_{n=0}^{\gamma_C-1}\epsilon\bigg]+E_x\bigg[\sum_{n=0}^{\gamma_C-1}b1_C(x_{\tau_n})\bigg]\bigg) \nonumber
\end{eqnarray}
and thus,
\[\sup_{x\in C}E_x[\gamma_C] \leq {L_1 + b \over \epsilon}.\]
Therefore $C\in \clB^+(\state)$ is a petite set such that $\sup_{x\in
  C}E_x\big[\sum_{n=0}^{\tau_C-1}r(n)\big]$ is finite and so $\{x_t\}$ satisfies Theorem \ref{Main
  Theorem Tweedie}(i) with $f=1$ and is $(1,r)$-ergodic. \end{proof}

\begin{remark}
  We note that just as  in the previous theorem,  if the stopping times satisfy
  Lemma~\ref{ch5: random sampling preserves smallness},  we can focus on return times
  for a petite set $A \subseteq \{ V \leq L\}$ with $\frac{V(x)}{W(x)}\leq
  V(x)-\epsilon$ for all $x\notin A$ instead of $C$. Similarly,  if the stopping
  times satisfy Lemma \ref{ch5: event-triggered case} and Assumption \ref{completeAssumption} we can focus on return
  times for a petite set $A$ in $\bigcap_n E_n$ with $\frac{V(x)}{W(x)}\leq
  V(x)-\epsilon$ for all $x\notin A \cap \big(\bigcap_n E_n\big)$ instead of $C$. This
  allows us to relax the conditions for $V$.
\end{remark}

As an example, with $r(n)=2 n^{\alpha}$, let for all $k$, $E[\sum_{m=\tau_k}^{\tau_{k+1}-1}
(m-\tau_k)^{\alpha} | {\cal F}_{\tau_k}] < \infty$ and $E[(\tau_{k+1}-\tau_k)^{\alpha}|{\cal F}_{\tau_k}] \leq \lambda^{-1}$. Then, the chain
is {\it polynomially ergodic}. Note that one can obtain explicit
expressions for a large class of sums of powers of the form $\sum_{k=0}^{\tau_1}
k^{\alpha}$ with $\alpha \in \mathbb{Z}_+$.

We also note that if $r$ satisfies $\sup_k E_x\big[r(\tau_{k+1}-\tau_k) \mid
x_{\tau_k}\big]\le  M$ for some finite $M$, then by Jensen's inequality $r^{1/s}$
satisfies the  bound  $\sup_k E\big[r^{1/s}(\tau_{k+1}-\tau_k)\mid \clF_{\tau_k}\big]\leq
\lambda^{-1}$ if $s>1$ is large enough so that $M^{1/s}\leq
\lambda^{-1}$.

Suppose now that the sequence of stopping times are state-dependent but
deterministic, that is $\tau_{k+1}=\tau_k + n(x_{\tau_k}), \tau_0=0$.

\begin{corollary}
Let $\{x_t\}$ be an aperiodic and irreducible Markov chain with a small set
  $C$. Suppose there exist $V: \state \rightarrow [1,\infty)$ which is
  bounded on $C$ and with for some $\epsilon>0$, $\lambda \in (0,1)$, $\lambda V(x) \leq V(x)-\epsilon$ for all $x\notin C$, and $b \in \mathbb{R}$ such that for an increasing sequence of stopping times $\{\tau_n\}$
  \begin{align}
E\big[V(x_{\tau_{n+1}}) \mid {\cal F}_{\tau_n}] \leq \lambda{V(x_{\tau_n})} +b1_C(x_{\tau_n}).
\end{align}
Then for any $r \in \Lambda$ and $M>0$ that satisfy
\[\sum_{k=0}^{n(x_0)} r(k) \leq M, \quad r(n(x_0)) \leq {1 \over \lambda}, \quad x_0 \in \state,\]
then  $\{x_t\}$ satisfies Theorem \ref{Main Theorem Tweedie} with
$f=1$ and it  is $(1,r)$-ergodic.
\end{corollary}

We note that Theorem \ref{ch5: r ergo with uni drift} above is useful for
proving $(1,r)$-ergodicity and Theorem \ref{Prop: random time app} is really
only useful for proving $(f,1)$-ergodicity, where $r,f$ satisfy the respective
hypotheses. In order to be able to prove more rate results, we may use results by
Douc et al.\ \cite{Douc} on the class $\mathcal{Y}$ of pairs of ultimately non-decreasing functions defined in Section~\ref{Sec: Subgeometric Ergodicity}. If a
Markov chain $\{x_t\}$ satisfies Theorem \ref{ch5: r ergo with uni drift} with
$(1,r)$ and Theorem \ref{Prop: random time app} with $(f,1)$, then $\{x_t\}$
is $(\Psi_1 \circ f, \Psi_2 \circ r)$-ergodic for
$(\Psi_1,\Psi_2)\in \mathcal{Y}$ by Proposition \ref{prop: Y ergo ext}.

%
%

Before ending this section, we revisit a criterion by Connor and Fort
\cite{ConnorFort} who studied rates of convergence  under drift
criteria which are based on state-dependent but deterministic sampling times so that
\[P^{n(x)}V(x)\leq \lambda V(x)+b1_C(x)\]
where $n:\state\rightarrow \mathbb{Z}_+$ is the state dependent time where the drift condition is enforced. Now, consider the case where $n$ is random and we have a sequence of stopping times
$\{\tau_k\}$ defined as $\tau_{k+1}=\tau_k+n(x_{\tau_k})$ with $\tau_0=0$. Theorem 3.2(i) of
\cite{ConnorFort} can be partly generalized to the random-time case as follows.

\begin{theorem}
  Let $\{x_t\}$ be an aperiodic and irreducible Markov chain with a small set
  $C$. Suppose that the stopping times $\{\tau_n\}$ satisfy
the conditions of Lemma~\ref{ch5: random sampling preserves smallness} and that there exist a function $V: \state \rightarrow [1,\infty)$, $V$
  bounded on $C$, and constants $b \in \mathbb{R}$ and $\lambda \in (0,1)$ such
  that for an increasing sequence of stopping times $\{\tau_n\}$ with $\tau_0=0$,
\[
E_x[V(x_{\tau_{k+1}})\mid \clF_{\tau_k}]\leq \lambda
V(x_{\tau_k})+b1_C(x_{\tau_k}).
\]
If there exists a strictly increasing function $R:(0,\infty)\rightarrow
(0,\infty)$ such that $R(t)/t$ is non-increasing and $E[R(\tau_{k+1}-\tau_k)\mid
x_{\tau_k}] \leq V(x_{\tau_k})$, then there exists a constant $D$ such that
$E_x[R(\tau_C)]\leq DV(x)$. If in addition the invariant distribution $\pi$ of
$\{x_t\}$ exists, $\pi(V)<\infty$, then the Markov chain is $(1,R)$-ergodic.
\end{theorem}

\begin{proof}
  Since $R(t)/t$ is non increasing, it follows that
  ${\log R(t) \over t} \to 0$ and $R \in \Lambda$. It also follows that $R(a+b)\leq
  R(a)+R(b)$ for any $a,b> 0$. With $R$ also increasing we have that
\begin{align*}
  E_x[R(\tau_{_C})] &\leq E_x[R(\tau_{\gamma_C})] =
  E_x\bigg[R(\sum_{k=0}^{\gamma_C-1}\tau_{k+1}-\tau_k)\bigg]\\
   & \leq
   E_x\bigg[\sum_{k=0}^{\gamma_C-1}R(\tau_{k+1}-\tau_k)\bigg] \\
   & = E_x\bigg[\sum_{k=0}^{\gamma_C-1}E_x\big[R(\tau_{k+1}-\tau_k)\mid
   \clF_{\tau_k}\big]\bigg]\\ &
 \leq E_x\bigg[\sum_{k=0}^{\gamma_C-1}V(x_{\tau_k})\bigg].
\end{align*}
With the drift condition $P^{\tau_{k+1}-\tau_k}V\leq V - (1-\lambda)V+b1_C$ and
$V$ bounded on $C$, we have
\[
(1-\lambda)E_x\bigg[\sum_{k=0}^{\gamma_C-1}V(x_{\tau_k})\bigg]\leq V(x)+b,
\]
where $\tau_0=0$ and so we obtain  $E_x[R(\tau_C)]<DV(x)$ for some $D>0$.

If the invariant distribution $\pi$ of $\{x_t\}$ exists and $\pi(V)<\infty$,
then by Theorem 14.2.11 of \cite{MeynBook} there exists a small set $A$ and $M \in \mathbb{R}$ such
that $\sup_{x\in A}E_x\big[\sum_{k=0}^{\tau_A-1}V(x_k)\big]<M$. Defining the
hitting time $\sigma_A=\min\{t\geq 0:x_t\in A\}$, the function
$W(x)=E_x\big[\sum_{n=0}^{\sigma_A}V(x_n)\big]$ satisfies the drift condition $PW\leq
W-V+M1_A$ with $A$ petite, and by Theorem \ref{theorem: petite set finite mean}
\begin{eqnarray}
&& E_x\bigg[\sum\limits_{k=0}^{\tau_B-1}V(x_k)\bigg] \nonumber \\
&& \leq W(x)+ME_x\bigg[\sum\limits_{k=0}^{\tau_B-1}1_A(x_k)\bigg]\leq W(x)+ M c(B) \nonumber
\end{eqnarray}

for any $B\in \clB^+(\state)$. Therefore, since $R$ is increasing,
\begin{eqnarray} \label{eq:W drift RT}
&& E_x\bigg[\sum\limits_{k=0}^{\tau_C-1}R(k)\bigg] \leq
E_x\bigg[\sum\limits_{k=0}^{\tau_C-1}E\big[R(\tau_C)\mid \clF_k\big]\bigg]  \nonumber \\
&& \leq
E_x\bigg[\sum\limits_{k=0}^{\tau_C-1}D V(x_k)\bigg]\leq D(W(x)+M c(C)).
\end{eqnarray}
To complete the proof, we show that $W$ is bounded on $C$. If the stopping times are independent and thus satisfy the conditions of
Lemma~\ref{ch5: random sampling preserves smallness}, then $C$ is petite for the
randomly sampled chain $\{x_{\tau_n}\}$ and the drift condition in the
hypothesis gives $E_x[R(\tau_B)]\leq (c(B) +1)V(x)$ for any $B\in
\clB^+(\state)$. Since $W$ satisfies a drift condition, $\{W< \infty\}$ is full
and absorbing and we can find a petite set in $\{W<\infty\}$.

Combining the above with (\ref{eq:W drift RT}) gives
\[
\sup_{x\in B}E_x\bigg[\sum\limits_{k=0}^{\tau_B-1}R(k)\bigg]\leq \sup_{x\in
  B} (c(B)+1)(W(x)+b)<\infty,
\]
for an appropriate petite set $B$ when the conditions of Lemma~\ref{ch5: random
  sampling preserves smallness} are satisfied. Thus  $\{x_t\}$ satisfies Theorem \ref{Main Theorem Tweedie}(i)
with  $(f,r)=(1,R)$ and it is  $(1,R)$ ergodic.
\end{proof}

\subsection{Geometric ergodicity }

We use the same reasoning as before to obtain geometric ergodicity from a random
time univariate drift condition.

\begin{theorem}\label{geoRandom}
  Let $\{x_t\}$ be an aperiodic and irreducible Markov chain with a small set
  $C$. If there exists a function $V: \state \rightarrow [1,\infty)$, $V$ bounded on $C$, constants
  $b\in \mathbb{R}$, $B>0$,  and $\lambda,\beta\in (0,1)$ such that for a sequence of stopping
  times $\{\tau_n\}$
\begin{align*}
  E[V(x_{\tau_{n+1}}) \mid {\cal F}_{\tau_n}] &\leq \lambda V(x_{\tau_n}) + b
  1_C(x_{\tau_n}),\\
  \intertext{and}
P(\tau_{n+1} - \tau_n =k \mid {x}_{\tau_n}) &\leq B \beta^k,
  \quad \text{for all $n,k$, and $x_{\tau_n} \notin C$}
\end{align*}
with \[{1 - B \lambda \over \beta} > 1,\]
and
\begin{eqnarray}\label{ilkBound}
\sup_{x \in C} E_x[a^{\tau_1}] < \infty
\end{eqnarray}
 for some $a > 1$, then ${x_t}$ is geometrically ergodic.
\end{theorem}

\begin{proof}
   By Theorem \ref{theorem: geometric ergodicity MT} for $r \in (1,\lambda^{-1})$
\[\sup_{x\in C} E_x[r^{\gamma_C}] < \infty. \]
Let $\rho \in (1, {1 - B \lambda \over \beta})$. Then,
\begin{eqnarray}
&&  E_x\big[\rho^{\tau_{n+1}-\tau_n} | {\cal F}_{\tau_n} \big] \leq {B \over 1 - \rho \beta} < \lambda^{-1}
\end{eqnarray}
for $x \notin C$. By a use of iterated expectations
\begin{equation}\label{ch5: expect rho tauA}
  E_x[\rho^{\tau_{\gamma_C}}]= E_x\big[ \prod_{n=0}^{\gamma_C-1} \rho^{\tau_{n+1} - \tau_n} \big] < E_x\big[\lambda^{-(\gamma_C-1)} \rho^{\tau_1} \big].
\end{equation}
By letting $1 < \rho < \min(a, {1 - B \lambda \over \beta})$, we obtain that $C\in \clB^+(\state)$ is a small set with a uniformly bounded $E_x[\rho^{\tau_{\gamma_C}}]$ for $x \in C$. Therefore by Theorem \ref{theorem: geometric ergodicity MT} the chain $\{x_t\}$ is geometrically ergodic.
 \end{proof}

 We also note that the rate of ergodicity relies on the constants $m$ and $\delta$
 for some $(m,\delta,\nu)$-small set $C$, so the ergodicity rate cannot be made
 explicit using only the information in the drift
 condition. 


\section{An Example in Networked Control}\label{SectionE}
We revisit the motivating example in \cite{YukMeynTAC2010}, concerning the stabilization problem over erasure channels. In particular, we apply the results of the previous section to establish a rate of convergence to equilibrium provided that the information transmission rate satisfies a certain inequality. We consider a scalar LTI discrete-time system described by
\begin{eqnarray}
\label{ProblemModel4CHP9}
x_{t+1}=ax_{t} + bu_{t} + w_t, \quad \quad t \ge 0,
\end{eqnarray}
where $x_t$ is the state at time $t$, $u_t$ is the control input, the initial
state $x_0$ is a random variable with a finite second moment, and $\{w_t \}$ is a
sequence of zero-mean i.i.d.\ Gaussian random variables, also independent of
$x_0$. We assume that the system is open-loop unstable and controllable, that
is, $|a| \geq 1$ and $b \neq 0$. This system is connected over a noisy channel
to a controller, as shown in Figure~\ref{LLL1}. The channel is assumed to have
finite input alphabet $\mathcal{M}$ and finite output alphabet $\mathcal{M}'$. A
source coder maps the source symbols (state values) to corresponding channel
inputs. The quantizer outputs are transmitted through the channel, after passing
through a channel encoder. The receiver has access to noisy versions of the
quantizer/coder outputs for each time instant $t$, which we denote by $q'_t \in
{\cal M}'$.

The problem is to identify conditions on the channel so that there exist coding and control schemes leading to the stochastic stability of the controlled process. For a thorough review of such problems with necessity and sufficiency conditions, see \cite{YukselBasarBook}.

\begin{figure}[h]
        \begin{center}
        \includegraphics[height=2.5cm,width=8cm]{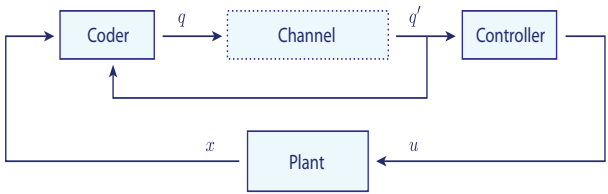}
\caption{Control of a system over a noisy channel. \label{LLL1}}
        \end{center}
\end{figure}
The source output is quantized
as follows:
\begin{eqnarray}
Q_K^{\Delta}(x) = \begin{cases}   (k - \half (K + 1) ) \Delta,
	 \nonumber \\
\quad \quad \quad \quad \quad \quad \mbox{if} x \in [ (k-1-\half K   ) \Delta , (k-\half K  ) \Delta)  \\
 (\half (K - 1) ) \Delta,  \quad \quad
	\mbox{if} \ \ x = \half K \Delta \\
  0 ,  \quad \quad
  	\mbox{if} \ \ x \not\in [- \half K  \Delta, \half K  \Delta]
\end{cases} \label{UniformQuantizerDescriptionCHP8}
\end{eqnarray}
where $K$ is a positive integer.  The quantizer outputs are transmitted through
a memoryless erasure channel, after being subjected to a bijective mapping,
which is performed by the channel encoder. At time $t$, the channel encoder
$\mathcal{E}_t$ maps
the quantizer output symbols to corresponding channel inputs $q_t\in
\clM\eqdef\{1,2\dots,K+1\}$ so  that ${\cal
  E}_t(Q_t(x_t))=q_t$. The controller/decoder has access to noisy
versions of the encoder outputs $q'_t \in
\mathcal{M}'\coloneqq \clM \cup \{e\}$, with $e$ denoting the erasure symbol, generated according to a
probability distribution for every fixed $q \in \clM$. The channel transition
probabilities are given by
\[
P(q'=i|q=i) = p, \quad \quad P(q'=e|q=i) = 1-p, \quad \quad i \in {\cal M}.
\]
At each time $t$, the controller/decoder applies a mapping
${\cal D}_t: \clM \cup \{e\} \to \mathbb{R}$, given by
\[{\cal D}_t(q'_t) = {\cal E}_t^{-1}(q'_t) \times 1_{\{q'_t \neq e\}} + 0 \times
1_{\{q'_t=e\}}.\] Let $\{\Upsilon_t\}$ denote the sequence of i.i.d.\ binary
random variables, representing the erasure process in the channel, where the
event $\Upsilon_t=1$ indicates that the signal is transmitted with no error
through the channel at time $t$. Let $p=\Expect[\Upsilon_t]$ denote the
probability of success in transmission. The following key assumptions are
imposed: Given $K\ge 2$ introduced in the definition of the quantizer, define
the \textit{rate variable}s
\begin{equation}
R = \log_2(K+1) \quad \quad R'=\log_2(K) \label{e:Rbdd}
\end{equation}
We fix positive scalars $\delta$ and $\alpha$ satisfying $|a| 2^{-R'} < \alpha <
1$ and $\alpha (|a|+\delta)^{p^{-1}-1} < 1.$ With $L>0$ a constant, let
$\bar{Q}: \mathbb{R} \times \mathbb{R} \times \{0,1\} \to \mathbb{R}$ be defined
as
\[
 \bar{Q}(\Delta,h,p) =
 \begin{cases}
   |a| + \delta, & \text{ if } \quad |h| > 1,  \mbox{ or } p = 0 \\
   \alpha,  &  \text{ if } 0 \leq |h| \leq 1,\;  p=1,\;  \Delta > L   \\
  1,  & \text{ if }  0 \leq |h| \leq 1,\; p=1, \; \Delta \leq L  .
\end{cases}
\]
For each $t \ge 0$ and with
$\Delta_0 \in \Re$ selected arbitrarily, let
\begin{eqnarray}
\label{QuantizerUpdate22}
u_t &=& - {a \over b} \hat{x}_{t}, \nonumber \\
\hat{x}_t &=& {\cal D}_t(q'_t) = \Upsilon_t Q_K^{\Delta_t}(x_t), \nonumber \\
\Delta_{t+1} &=& \Delta_t \bar{Q}(\Delta_t,|{ x_{t} \over \Delta_t 2^{R'-1} }|,\Upsilon_t).
\end{eqnarray}
Given the channel output $q'_t \neq e$, the controller can simultaneously deduce
the realization of $\Upsilon_t$ and the event $\{|h_t| > 1\}$, where $h_t =
\frac{ x_t}{ \Delta_t 2^{R'-1}}$. This is due to the fact that if the channel
output is not the erasure symbol, the controller knows that the signal is
received with no error. If $q'_t=e$, however, then the controller applies $0$ as
its control input and enlarges the bin size of the quantizer.

By Lemma 3.1 of \cite{YukMeynTAC2010}, $(x_t, \Delta_t)$ is a Markov chain.

Consider now a sequence of stopping times which denote the times when there is a successful transmission of a source symbol in the {\it granular region} of the quantizer:
\begin{eqnarray}
\tau_0 = 0, \quad \tau_{z+1} &=& \inf \{k > \tau_z : |h_{k}| \leq  1, p_k=1 \}, \quad z \in \mathbb{Z}_+ \nonumber 
\end{eqnarray}

By Proposition 3.1 of \cite{YukMeynTAC2010}, the stopping time distribution is bounded uniformly by a geometric measure:

\begin{lemma}[\protect{\cite[Proposition 3.1]{YukMeynTAC2010}}]\label{stopBound}
  The discrete probability measure $\Prob(\tau_{i+1}-\tau_i=k\mid x_{\tau_i
  },\Delta_{\tau_i })$ satisfies,
\[
(1-p)^{k-1} \leq \Prob(\tau_{i+1}-\tau_i \geq k \mid x_{\tau_i },\Delta_{\tau_i }) \leq (1-p)^{k-1} +  o(1),
\]
where $o(1) \to 0$ as $\Delta_{\tau_i} \to \infty$ uniformly in $x_{\tau_i}$.
\end{lemma}

As a consequence, the probability $\Prob(\tau_{i+1}-\tau_i \geq k \mid x_{\tau_i },\Delta_{\tau_i })$ tends to $ (1-p)^{k-1} p$ as $\Delta_{\tau_i} \to \infty$. 

\begin{theorem}
\label{FiniteMoment}
Suppose that
\begin{eqnarray}\label{2ndMomentNec}
	a^2 \Bigl(1-p + {p \over (2^{R}-1)^2} \Bigr) < 1.
\end{eqnarray}
Then, under the coding and control policy considered, the chain $(x_t, \Delta_t)$ is geometrically ergodic.
\end{theorem}

\begin{proof}
By the proof of Theorem 3.2 of \cite{YukMeynTAC2010}, with $V(x,\Delta)=\Delta^2$, and
 with \[0 < \epsilon < 1 - {p \alpha^2 \over 1 - (1-p)(|a|+\delta)^2},\] $b< \infty$ and $C$ a small set, it follows that
\begin{eqnarray}
&& \Expect[V(x_{\tau_{z+1}},\Delta_{\tau_{z+1}}) \mid x_{\tau_z}]  \nonumber \\
&& \leq (1 - \epsilon)  V(x_{\tau_{z}},\Delta_{\tau_{z}}) + b1_{\{(x_{\tau_{z}},\Delta_{\tau_{z}}) \in C\}}.
\end{eqnarray}
Together with Lemma \ref{stopBound}, and (34) in \cite{YukMeynTAC2010} leading to (\ref{ilkBound}), these imply that for some $\lambda=1-\epsilon \in (0,1)$, $B \in (p, {p \over \lambda})$ for some sufficiently large $C$, and $\beta = 1-p$, Theorem \ref{geoRandom} holds. 
\end{proof}

\begin{remark}
We recall from \cite{YukMeynTAC2010} that under (\ref{2ndMomentNec}), the system is quadratically stable in the sense that for each initial condition $(x_0,\Delta_0)$, $\lim_{t \to \infty} \Expect[x_t^2]  = \Expect_\pi[x_0^2] < \infty$. We also note that by \cite{YukMeynTAC2010} if the goal is to only have the existence of an invariant probability measure, the requirements on the channel reduce to the conditions that $|a| 2^{-R'} < \alpha < 1$ and $\alpha (|a|+\delta)^{p^{-1}-1} < 1.$
\end{remark}
%

\section{Conclusion}

In this paper, we established random-time state-dependent drift criteria for
Markov chains using Lyapunov-theoretic methods. We established drift criteria
both for sub-geometric and geometric rates of convergence, where the conditions revealed the relationship between the distributions of the stopping times, the drift of the Lyapunov functions at random times, and the ergodicity rates. Future work
includes the application of these results in event triggered control systems, as
well as information theory problems for variable-length decoding.


\end{document}